\documentclass[11pt]{article}
\usepackage{amsfonts}
\usepackage{amssymb}
\usepackage{amsthm}								 
\usepackage{latexsym}
\usepackage{ae}                    
\usepackage{tabularx}              
\usepackage{enumerate}             
\usepackage{enumitem}							 
\usepackage{bbm}									 
\usepackage[arrow,matrix,curve]{xy}
\usepackage{makeidx}
\usepackage[final,dvips]{graphicx}
\usepackage{fancyhdr}
\usepackage[latin1]{inputenc}
\usepackage{amsmath}

\newtheoremstyle{SatzmitAbstand}
{10pt}    
{10pt}    
{\itshape}     
{25pt}    
{\bfseries} 
{:}     
{5pt} 
{}     
 
\theoremstyle{SatzmitAbstand}

\newtheorem{theorem}{Theorem}[section] 
\newtheorem{proposition}[theorem]{Proposition}
\newtheorem{corollary}[theorem]{Corollary}

\newtheorem{lemma}[theorem]{Lemma}

\newtheorem{definition}[theorem]{Definition}

\newtheorem*{theorem*}{Theorem}

\newcommand{\HH}{\mathcal{H}}
\newcommand{\ZZ}{\mathbb{Z}}
\newcommand{\KK}{\mathcal{K}}
\newcommand{\MM}{\mathcal{M}}
\newcommand{\CC}{\mathbb{C}}

\newcommand{\NN}{\mathbb{N}}

\newcommand{\BH}{\mathcal{B}(\mathcal{H})}
\newcommand{\CZK}{C(\prod_{j=1}^\infty S^2,\KK)}
\newcommand{\MCZK}{\MM(C(\prod_{j=1}^\infty S^2,\KK))}
\newcommand{\norm}[1]{\left\|#1\right\|}

\begin{document}


\begin{center}
{\bf \huge On certain multiplier projections}
\end{center}

\vspace{0.3cm}

\begin{center}
{\it \large Henning Petzka}
\end{center}

\vspace{0.3cm}

\begin{abstract}
\noindent Let $\MCZK$, denote the multiplier algebra over $\CZK$, the algebra of continuous functions into the compact operators with spectrum the infinite product of two-spheres. We consider multiplier projections in $\MCZK$ of a certain diagonal form. We show that, while for each multiplier projection $Q$ of the special form, we have that $Q(x)\in\BH\setminus \KK$ for all $x\in \prod_{j=1}^\infty S^2$, the ideal generated by $Q$ in $\MCZK$ might be proper. We further show that the ideal generated by a multiplier projection of the special form is proper if and only if the projection is stably finite.
\end{abstract}

\section{Introduction}
The C*-algebra $\CZK$ of continuous functions into the compact operators with spectrum the infinite product of two-spheres has been of interest in the construction of C*-algebras with non-regular behavior. M. R\o rdam used this algebra in \cite{R1} to construct a separable simple C*-algebra with both a finite and a (non-zero) infinite projection. In \cite{R3} R\o rdam used it to construct an extension 
$$0\rightarrow \CZK\rightarrow B\rightarrow \KK\rightarrow 0$$
such that $B$ is not stable (despite the fact that both, ideal and quotient, are stable C*-algebras). Also, R\o rdam's construction  in \cite{R2} of  a non-stable C*-algebra, which becomes stable after tensoring it with large enough (non-zero) matrix algebras, can be altered to using comparability properties of projections in matrix algebras over $\CZK$.\\

\noindent Most constructions have in common that they take advantage of special multiplier projections of a certain diagonal form. The projections considered are infinite direct sums 
\begin{equation*}\tag{$\ast$}\label{*}
Q=\bigoplus_{j=1}^\infty p_{I_j}, 
\end{equation*}
where each direct summand $p_{I_j}$ is a finite tensor product of Bott projections over coordinates specified by a finite subset $I_j$ of the natural numbers. (We remind the reader of the detailed construction in the following section.)  Using multiplier projections of this certain form, R\o rdam proves in \cite{R1} that there exists a finite full multiplier projection in $\MCZK$ (and thereby showing that the C*-algebra $\CZK$ does not have the corona factorization property).  Recall that a projection in a C*-algebra is called full, if the closed two-sided ideal generated by it is the whole C*-algebra. Fullness of a projection in the multiplier algebra implies that some multiple of it is equivalent to the identity (\cite{RLL}, Exercise 4.8). But the multiplier unit of a stable C*-algebra is properly infinite (\cite{R5}, Lemma 3.4). Hence, R\o rdam's finite full projection is not stably finite. (A projection is stably finite if any multiple of it is a finite projection).\\

\noindent In this paper we investigate non-full multiplier projections in $\MCZK$ of the special form as in ($\ast$). Firstly, it is all but obvious that there exist non-full projections of this diagonal form at all. Identifying $\MCZK$ with the strictly continuous functions from $\prod_{j=1}^\infty S^2$ into $\BH$, any multiplier projection $Q$ of the certain diagonal form satisfies that $Q(x)\in\BH\setminus \KK$. In particular, locally, $Q(x)$ is full in $\BH$ for all $x\in X$.(It follows from the results of Pimsner, Popa and Voiculescu \cite{PPV} that such a projection cannot be found when the spectrum is finite-dimensional.)\\

\noindent Using the techniques from \cite{R1} we then prove the following result:
\begin{theorem*}
Let $$Q=\bigoplus_{j=1}^\infty p_{I_j}\in \MCZK.$$
Then $Q$ is non-full if, and only if, $Q$ is stably finite.
\end{theorem*}

The paper is organized as follows. In Section \ref{preliminaries} we recall notation and constructions from \cite{R1} and specify the multiplier projections the paper is devoted to. Section \ref{technicalresult} contains the technical tool to prove our main results. In Section \ref{sectionnonfull} we characterize non-fullness of multiplier projections in a combinatorial way. Finally, Section \ref{sectionstablyfinite} contains the proof of the main theorem, i.e. we show that all non-full projections from section \ref{sectionnonfull} are stably finite.

\section{Preliminaries}\label{preliminaries}

Consider the following setting (and notation), which is adapted from \cite{R1}. We will consider the Hausdorff space given by an infinite product of two-spheres, $X=\prod_{j=1}^\infty S^2$, equipped with the product topology. Since $S^2$ is compact, it follows from Tychonoff's Theorem (see for example \cite{M}) that $X$ is compact. Let further 
$$p\in \mbox{C}(S^2,M_2(\CC)) \mbox{ denote the Bott projection,}$$
i.e., the projection corresponding to the `Hopf bundle' $\xi$ over $S^2$ with total Chern class $c(\xi)=1+x$ (see e.g. \cite{K}). 

With $\pi_n:X\rightarrow S^2$ denoting the coordinate projection onto the $n$-th coordinate, consider the (orthogonal) projection
$$p_n:=p\circ \pi_n \in \mbox{C}(\prod_{j=1}^\infty S^2,M_2(\CC)).$$
If $I\subseteq \NN$ is a finite subset, $I=\{n_1,n_2,\ldots, n_k\}$, then let $p_I$ denote the pointwise tensor product 
$$p_I:=p_{n_1}\otimes p_{n_2} \otimes \ldots \otimes p_{n_k}\ \in \mbox{C}\left (\prod_{j=1}^\infty S^2,M_2(\CC)\otimes M_2(\CC)\otimes\ldots \otimes M_2(\CC)\right ).$$

It is shown in \cite{R1} that the projection $p_n$ corresponds to the pull-back of the Hopf bundle via the coordinate projection $\pi_n$, denoted by $\xi_n:=\pi_n^*(\xi)$, and that the projection $p_I$ corresponds to the tensor product of vector bundles $\xi_{n_1}\otimes \xi_{n_2}\otimes \ldots \otimes \xi_{n_k}$.

Considering the compact operators $\KK$ on a separable Hilbert space as an AF algebra, the inductive limit of the sequence
$$\CC\rightarrow M_2(\CC)\rightarrow M_3(\CC)\rightarrow M_4(\CC)\rightarrow \ldots ,$$
with connecting $^*$-homomorphisms mapping each matrix algebra into the upper left corner of any larger matrix algebra at a later stage, we get an embedding of each matrix algebra over $\CC$ into the compact operators $\KK$. In this way we can consider all the projections $p_n$ and $p_I$, defined as above, as projections in C$(\prod_{j=1}^\infty S^2,\KK)$.

In addition to the setting of \cite{R1}, let us denote by $p^{-}$ the projection corresponding to the complex line bundle $\xi^{-}$ over $S^2$ with total Chern class $c(\xi^-)=1-x$. (Recall that the first Chern class is a complete invariant for complex line bundles (see Proposition 3.10 of \cite{H}).) The tensor product $\xi\otimes\xi^-$ is isomorphic to the one-dimensional trivial bundle, because its Euler class can be computed, using \cite{R1} Equation 3.3, to be
$$e(\xi\otimes\xi^-)=x-x=0$$
and the only line bundle with zero Euler class is the trivial bundle $\theta_1$ (\cite{H} Proposition 3.10). Accordingly, the projection in $C(X,M_4(\CC))$  that is given by the pointwise tensor product of $p$ and $p^{-}$ is equivalent to a 1-dimensional constant projection.

We finally define $p_n^{-}\in C(\prod_{j=1}^\infty S^2,\KK)$ as $p_n^{-}:=p^{-}\circ \pi_n$.\\

\noindent The following well known result can be found for example in \cite{L}.

\begin{lemma}\label{multiplierarefunctionsintobh}
Let $X$ be a compact Hausdorff space. Let $C(X,\KK)$ denote the continuous functions from $X$ into the compact operators $\KK$ on a separable Hilbert space $\mathcal{H}$. Let further $C_{*\mbox{-}s}(X,\BH)$ denote the  $^*$-strongly continuous (or strictly continuous) functions from $X$ into the bounded operators on the Hilbert space $\mathcal{H}$. Then
$$\mathcal{M}(C(X,\KK))\cong C_{*\mbox{-}s}(X,\BH).$$
\end{lemma}

We will often take advantage of identifying $C(X,\KK)$ with $C(X)\otimes \KK$ . For instance it is then immediate to see stability of $C(X,\KK)$.\\


\noindent For any stable C*-algebra $A\cong A\otimes \KK(\HH)$ we can embed the algebra of all bounded operators $$\BH\cong \mathbbm{1}_{\mathcal{M}(A)}\otimes\BH$$ into $\mathcal{M}(A)$ (see e.g. \cite{L} Chapter 4).  Hence, we can find a sequence $\{S_j\}_{j=1}^\infty$ of isometries with orthogonal range projections in $\mathcal{M}(A)$ such that the range projections sum up to the identity of $\mathcal{M}(A)$ in the strict topology:
$$S_j^*S_j=\mathbbm{1}_{\mathcal{M}(A)}\mbox{ for all $j$, and } \sum_{j=1}^\infty S_jS_j^*=\mathbbm{1}_{\mathcal{M}(A)}.$$
Using such a sequence we can define infinite direct sums of projections in $A$: 

For a sequence $\{p_j\}_{j=1}^\infty$ of projections in $A$  we define 
$$\bigoplus_{j=1}^\infty p_j:=\sum_{j=1}^\infty S_j p_j S_j^*\in\mathcal{M}(A).$$
The sum is strictly convergent and hence defines a projection in the multiplier algebra of $A$, which, up to unitary equivalence, is independent of the chosen isometries (\cite{R1}, page 10). Also, its unitary equivalence class in the ordered Murray-von Neumann semigroup is independent of permutations of the direct summands (see Lemma 4.2 of \cite{R1}).

For fixed projections $Q\in\mathcal{M}(A)$ we will denote the direct sum $\underbrace{Q\oplus Q\oplus \ldots \oplus Q}_{\mbox{$m$ times}}$ of $Q$ with itself by $m\cdot Q$. \\

\noindent We are now ready to specify the multiplier projections this paper is devoted to and which were considered by R\o rdam in \cite{R1} and \cite{R3}: All our results are for multiplier projections given by
\begin{equation*}\tag{$\ast$}
Q=\bigoplus_{j=1}^\infty p_{I_j},
\end{equation*}
where each $p_{I_j}$ is a tensor product of Bott projections as above.

\section{Technical result}\label{technicalresult}

The following result is basically contained in \cite{R1} by a combination of Proposition 3.2 with Proposition 4.5 from that paper. It makes it possible to check minorization of projections as in ($\ast$) by trivial projections in C$(\prod_{j=1}^\infty S^2,\KK)$ in purely combinatorial terms. By a trivial projection we mean a projection that is equivalent to a constant one (i.e., any projection that corresponds to a trivial complex vector bundle). We will denote trivial 1-dimensional projections in C$(\prod_{j=1}^\infty S^2,\KK)$  by $g$. Recall that for any non-empty finite subset $I$ of $\NN$ we denote by $p_I$ the tensor product of Bott projections over the coordinates given by $I$.

\begin{proposition}\label{infiniteeisnotsubprojection}
Let $I_j,\ j\in\NN$, be finite subsets of $\NN$, and consider the multiplier projection $Q$ in $\mathcal{M}(C(\prod_{j=1}^\infty S^2,\KK))$ given by
$$Q=\bigoplus_{j=1}^\infty p_{I_j}.$$
Then the following statements are equivalent:
\begin{itemize}
\item[(i)] $g \npreceq Q=\bigoplus_{j=1}^\infty p_{I_j}$.
 \item[(ii)] $\left |\bigcup_{j\in F}I_j \right |\geq |F|$ for all finite subsets $F\subseteq \NN$.
\end{itemize}
\end{proposition}

\begin{proof}
That (ii) implies (i) is the content of Proposition 4.5 (i) of \cite{R1}.

If, on the other hand, there is a finite subset $F\subseteq \NN$ such that $\left |\bigcup_{j\in F}I_j \right |< |F|$, consider the subprojection $\bigoplus_{j\in F} p_{I_j}$ in $C(\prod_{j=1}^\infty S^2,\KK)$. Let $J:=\bigcup_{j\in F}I_j$. With $\pi_J$ denoting the projection onto the coordinates given by $J$, we have $\bigoplus_{j\in F} p_{I_j} = \pi_J^*(q )$ for some projection $q\in C(\prod_{j=1}^{|J|}S^2,\KK)$. The projection $q$ corresponds to a vector bundle of dimension $|F|$ over $|J|=\left |\bigcup_{j\in F}I_j \right |$-many copies of $S^2$. But then by \cite{Hu}, Theorem 8.1.2, this vector bundle majorizes a trivial bundle. In terms of projections this implies $$g=\pi_J^*(g)\preceq \pi_J^*(q)= \bigoplus_{j\in F} p_{I_j}\preceq Q$$
\end{proof}

\noindent It is possible to generalize this result. The following proposition allows to count the precise number of trivial subprojections (while Propositon \ref{infiniteeisnotsubprojection} is only good to check existence of some trivial subprojection).

\begin{proposition}\label{infiniteexactnumberoftrivialsubprojections}
Let $I_j,\ j\in\NN$, be finite subsets of $\NN$, and consider the multiplier projection $Q$ in $\mathcal{M}(C(\prod_{j=1}^\infty S^2,\KK))$ given by
$$Q=\bigoplus_{j=1}^\infty p_{I_j}.$$
Let $m\in\NN$. \\
Then the following statements are equivalent:
\begin{itemize}
\item[(i)] $m\cdot g\npreceq Q\sim \bigoplus_{j=1}^{\infty}p_{I_j}$.
\item[(ii)] $|F| < \left | \bigcup_{j\in F} I_j \right |+m$ for all finite subsets $F\subseteq\NN$. 
\end{itemize}
\end{proposition}

\begin{proof}
The implication from (i) to (ii) can be seen from standard stability properties of vector bundles, as follows: Assume there is some finite subset $F$ such that $$|F| \geq \left | \bigcup_{j\in F} I_j \right |+m.$$ Then $\bigoplus_{j \in F}p_{I_j}$ is an $|F|$-dimensional subprojection of $Q$ that can be considered, using the identification of projections with vector bundles and using a pullback by the appropriate coordinate projection (as in the proof of Proposition \ref{infiniteeisnotsubprojection}), as an $|F|$-dimensional vector bundle over a base space consisting of the product of $\left | \bigcup_{j\in F} I_j \right |$ copies of $S^2$. Then Theorem 8.1.2 from \cite{Hu} proves the existence of a trivial $\left (|F|- \left | \bigcup_{j\in F} I_j \right |\right )$-dimensional subbundle. This implies  (again in terms of projections in $\mathcal{M}(C(\prod_{j=1}^\infty S^2,\KK))$):
$$m\cdot g\leq \left (|F|- \left | \bigcup_{j\in F} I_j \right |\right )\cdot g\preceq \bigoplus_{j\in F}p_{I_j}\leq \bigoplus_{j=1}^{\infty}p_{I_j}=Q.$$
\vspace{0.2cm}

Let us now prove that (ii) implies (i): By hypothesis all finite subsets $F\subseteq\NN$ satisfy 
$$|F| <\left | \bigcup_{j\in F} I_j \right |+m.$$
Assume $m\cdot g\preceq Q$. Then $m\cdot g \preceq \bigoplus_{j=1}^N p_{I_j}$  for some $N\in \NN$ by Lemma 4.4 of  \cite{R1}. Let $k_1,k_2,\ldots,k_{m-1}$ be natural numbers in $\NN\setminus \bigcup_{j=1}^N I_j.$ Then 
by Lemma 2.3 of \cite{KN} 
there exists a projection $q$ such that $$q\oplus \left (p_{k_1}^{-}\otimes p_{k_2}^{-}\otimes \ldots \otimes p_{k_{m-1}}^{-}\right )\sim m\cdot g\preceq  Q.$$
Tensoring (pointwise) both sides by $p_K:=p_{k_1}\otimes p_{k_2}\otimes \ldots \otimes p_{k_{m-1}}$, it follows that
$$\left  (q\otimes p_{k_1}\otimes p_{k_2}\otimes \ldots \otimes p_{k_{m-1}}\right )\oplus g\preceq \bigoplus_{j=1}^{\infty} p_{I_j}\otimes p_{k_1}\otimes p_{k_2}\otimes \ldots \otimes p_{k_{m-1}}.$$
In particular,
$$g\preceq \bigoplus_{j=1}^{\infty}p_{I_j}\otimes p_K=\bigoplus_{j=1}^{\infty}p_{I_j\cup K}.$$
By Proposition \ref{infiniteeisnotsubprojection} this entails that there is some finite subset $F\subseteq\NN$ such that
$$\left | \bigcup_{j\in F} I_j\cup K \right | < |F|.$$
Hence, 
$$|F|> \left | \bigcup_{j\in F} I_j\cup K \right | = \left | \bigcup_{j\in F} I_j \right | + |K|=\left | \bigcup_{j\in F} I_j \right |+(m-1).$$
But the existence of a finite subset $F$ satisfying
$$|F|\geq \left | \bigcup_{j\in F} I_j \cup K \right |+1= \left | \bigcup_{j\in F} I_j \right |+m$$ contradicts the hypothesis.
\end{proof}

\noindent If we want to consider multiples of the multiplier projection as well, we can apply

\begin{corollary}\label{exactnumberoftrivialsubprojectionsmultiple}
Let $I_j,\ j\in\NN$, be finite subsets of $\NN$, and consider the multiplier projection $Q$ in $\mathcal{M}(C(\prod_{j=1}^\infty S^2,\KK))$ given by
$$Q=\bigoplus_{j=1}^\infty p_{I_j}.$$
Let $m,n\in\NN$. \\
Then the following statements are equivalent:
\begin{itemize}
\item[(i)] $m\cdot g\npreceq n\cdot Q\sim \bigoplus_{j=1}^{\infty }n\cdot p_{I_j}$.
\item[(ii)] $n|F| < \left | \bigcup_{j\in F} I_j \right |+m$ for all finite subsets $F\subseteq\NN$.
\end{itemize}
\end{corollary}

\begin{proof}
Note, that in $n\cdot Q$ each index set $I_j$ appears $n$ times. Choosing the same set $I_j$ several times does not increase the left-hand side of the inequality (ii) of Proposition \ref{infiniteexactnumberoftrivialsubprojections}, while it does increase the right-hand side of that inequality. Now the statement follows immediately from Proposition \ref{infiniteexactnumberoftrivialsubprojections}. 
\end{proof}

\section{Non-full multiplier projections}\label{sectionnonfull}

The combinatorial description of subequivalences makes it possible to prove the following useful result.

\begin{lemma}\label{Ngsmallerenough}
If $N\cdot g\preceq \bigoplus_{j=1}^\infty p_{I_j}$ for all $N\in\NN$, then
$$\mathbbm{1}\preceq Q.$$
\end{lemma}

\begin{proof}
By Proposition \ref{infiniteexactnumberoftrivialsubprojections} the hypothesis is equivalent to:

For all $N\in\NN$ there is some finite subset $F\subseteq\NN$ such that
\begin{equation*} |F|\geq \left |\bigcup_{j\in F}I_j\right | +N.
\tag{$\ast \ast$}\end{equation*}
Let $G\subseteq \NN$ be any finite subset of the natural numbers. We claim that there is then some finite subset $H\subseteq (\NN\setminus G)$ such that  
$$g\preceq \bigoplus_{j\in H} p_{I_j}.$$

To show this, apply the hypothesis ($\ast \ast$) to the choice $|G|+1$ for $N$: we obtain a finite subset $F\subseteq\NN$ such that
$$|F|\geq \left |\bigcup_{j\in F }I_j\right | +|G|+1.$$
Then 
$$\left | \bigcup_{j\in F\setminus G}I_j\right |+1\leq \left | \bigcup_{j\in F}I_j\right |+1\leq |F|-|G|.$$
By Proposition \ref{infiniteeisnotsubprojection} this implies that
$$ g\preceq \bigoplus_{j\in F\setminus G} p_{I_j},$$
and we can take $H:=F\setminus G$.

Using this intermediate result we begin to iterate:\\
Firstly by assumption we have $g\preceq \bigoplus_{j=1}^\infty p_{I_j}$, and therefore by Lemma 4.4 of \cite{R1},
$$g\preceq \bigoplus_{j=1}^{m_1-1}p_{I_j}$$ for some $m_1\in\NN$. But then with $G=\{1,2,\ldots, (m_1-1)\}$ we can find, by application of the proven claim, some natural number $m_2>m_1$ and $H\subseteq \{m_1,(m_1+1),\ldots,(m_2-1)\}$ such that 
$$g\preceq \bigoplus_{j=m_1}^{m_2-1}p_{I_j}.$$
Iterating, we get a strictly increasing sequence of natural numbers $1=m_0<m_1<m_2<m_3<\ldots $ and, for all $i\in\NN$, we get a partial isometry $v_i\in C(\prod_{j=1}^\infty S^2,\KK)$ such that 
$$g=v_i^*v_i\sim v_i v_i^*\leq  \bigoplus_{j=m_{i-1}}^{m_i-1}p_{I_j}.$$
The multiplier 
$$V=\bigoplus_{i=1}^\infty v_i$$
then implements the subequivalence
$$\mathbbm{1}\sim \infty\cdot g\preceq \bigoplus_{j=1}^\infty p_{I_j}=Q.$$
\end{proof}

\noindent We can now prove the main theorem of this section, which is a combinatorial characterization for multiplier projections of the special form to be non-full.

\begin{theorem}\label{nonfull}
Let $Q=\bigoplus_{j=1}^{\infty}p_{I_j}\in \mathcal{M}(C(\prod_{j=1}^\infty S^2,\KK))$ be as above. Then the following statements are equivalent:
\begin{itemize}
\item[(i)] Q is non-full.
\item[(ii)] $\forall m\in\NN\ \exists N(m)\in \NN\ such\ that \ N(m)\cdot g\npreceq m\cdot Q$.
\item[(iii)] $\forall m\in\NN\ \exists N(m)\in \NN\ such\ that \ m|F| <|\bigcup_{j\in F}I_j|+N(m) \mbox{ for all finite subsets }F\subseteq \NN$.
\end{itemize}
\end{theorem}

\begin{proof}
The equivalence between (ii) and (iii) follows from Proposition \ref{exactnumberoftrivialsubprojectionsmultiple}.

If we are in the situation of the condition (ii), then, in particular, $\mathbbm{1}\npreceq m\cdot Q$ for any natural number $m$ and so $Q$ cannot be full (see \cite{RLL}, Exercise 4.8). This proves that (ii) implies (i).

Finally assume that there exists some $m\in\NN$ such that for all $N\in\NN$ we have $N\cdot g\preceq m\cdot Q$. Then by Lemma \ref{Ngsmallerenough} also $\mathbbm{1}\preceq m\cdot Q$ and $Q$ is full. So (i) implies (ii).
\end{proof}

\noindent Rephrasing the content of Theorem \ref{nonfull} we get the following interesting result.

\begin{corollary}\label{nonfullprojectionexample}
There exists a compact Hausdorff space $X$ and a projection $Q$ in $C_{*\mbox{-}s}(X,\BH)$, the multiplier algebra of $C(X,\KK)$, such that $Q(x)\in \BH\setminus\KK$ for all $x\in X$, and $Q$ is not full in $C_{*\mbox{-}s}(X,\BH)$.
\end{corollary}

In particular, the projection $Q(x)$ is full in the fiber over each $x\in X$, but $Q$ is itself non-full. It follows from the results of Pimsner, Popa and Voiculescu in \cite{PPV} that for obtaining such an example the space $X$ is necessarily of infinite dimension. 

\begin{proof}
Let $X=\prod_{j=1}^\infty S^2$. To show existence of the projection $Q$, choose pairwise disjoint subsets $I_j\subseteq\NN$ such that $|I_j|=n$ and set 
$$Q:=\bigoplus_{j=1}^{\infty}p_{I_j} \in \mathcal{M}(C(X,\KK)).$$
We then have that $Q(x)\in\BH\setminus \KK$, since $\norm{p_j(x)}=1$ for all $x\in X$ and all $j\in \NN$ (and since a strictly convergent sum of pairwise orthogonal elements in the compact operators $\KK$ belongs to $\KK$ if, and only if, the elements converge to $0$ in norm (cf. \cite{R1} Proof of Proposition 5.2)). So we only need to show that the index sets $I_j$ satisfy the condition (iii) of Theorem \ref{nonfull}; that is, we need to show that
$$\forall m\in\NN\ \exists N(m)\in \NN\mbox{ such that } m<\frac{|\bigcup_{j\in F}I_j|+N(m)}{|F|} \mbox{ for all finite subsets }F\subseteq \NN.$$
Now
$$\frac{\left | \bigcup_{j\in F }I_j\right |+\frac{m(m-1)}{2}}{|F|}\geq \frac{ \sum_{j=1 }^{|F|}j +\frac{m(m-1)}{2}}{|F|}$$
$$=\frac{ \frac{|F|(|F|+1)}{2} +\frac{m(m-1)}{2}}{|F|}$$
$$=\frac{1}{2}\left ( 1+ |F| + \frac{m(m-1)}{|F|}\right )$$
and the last expression is minimized when $|F|\in\{(m-1),m\}$.

Hence,
$$\frac{\left | \bigcup_{j\in F }I_j\right |+\frac{m(m-1)}{2}+1 }{|F|}> \frac{\left | \bigcup_{j=1 }^{m-1}I_j \right |+\frac{m(m-1)}{2}}{m-1}$$
$$= \frac{\frac{m(m-1)}{2}+ \frac{m(m-1)}{2}}{m-1}=m.$$
So we can choose $N(m)=\frac{m(m-1)}{2}+1.$
\end{proof}

\section{Stably finite multiplier projections}\label{sectionstablyfinite}

In this section we will show that every multiple of a non-full projection
$$Q=\bigoplus_{j=1}^\infty p_{I_j}$$
constructed as in Theorem \ref{nonfull} above (and, in particular, every multiple of the explicit projection of Corollary \ref{nonfullprojectionexample}), is a finite projection. In fact, our results show that a multiplier projection $Q$ of the special form is non-full if, and only if, it is stably finite (Corollary \ref{nonfullequalsstablyfinite}).\\

\noindent It is fairly easy to see that the projections $m\cdot Q$, where $Q$ is one of the non-full projections from Theorem \ref{nonfull}, cannot be properly infinite. This follows from the following lemma, together with the existence (Theorem \ref{nonfull}) of a number $N(m)\in\NN$ such that $$N(m)\cdot g\npreceq m\cdot Q,\mbox{ but } N(m) \cdot g\preceq  l\cdot Q\mbox{ for sufficiently large $l$}.$$
 
\begin{lemma}\label{notproperlyinfinite}
Let $A$ be a C*-algebra and $p$ and $q$ two projections in $A \otimes \KK$ such that $p\preceq k\cdot q$, but $p \npreceq m\cdot q$ for some $m<k$. Then $q$ is not properly infinite.
\end{lemma}

\begin{proof}
Assume that $q$ is properly infinite. Then
$$p \preceq k\cdot q=\underbrace{(q\oplus q)\oplus q\oplus \ldots \oplus q}_{k\ many}$$
$$\preceq \underbrace{(q)\oplus q\oplus \ldots \oplus q}_{(k-1)\ many}= (k-1)\cdot q$$
$$\preceq (k-2)\cdot q \preceq\ldots\preceq m\cdot q,$$
in contradiction with the assumption.
\end{proof}

\noindent It does not seem possible to see finiteness of these projections in a similarily easy way. To show finiteness we will need to give a somewhat complicated proof). The idea is the content of the following lemma and is essentially contained in the proof of Theorem 5.6 of \cite{R1}.

\begin{lemma}\label{finitenessargument}
Let $B$ be a simple inductive limit C*-algebra,
$$\xymatrix{B_1 \ar[r]^{\varphi_1}\ar@/_1.5pc/[rrr]_{\varphi_{i,1}}  & B_2 \ar[r]^{\varphi_2}&\ldots \ar[r] &B_i \ar[r]^{\varphi_i} & \ldots \ar[r] & B }$$
with injective connecting $^*$-homomorphisms $\varphi_j$. Let $q$ be a projection in $B_1$.

If the image $\varphi_{i,1}(q)$ of the projection $q$ is not properly infinite in any building block algebra $B_i$, then $q$ must be finite.
\end{lemma}

\begin{proof}
The hypothesis that $\varphi_{i,1}(q)\in B_i$ is not properly infinite for any $i\in\NN$ together with  Proposition 2.3 of \cite{R1} applied to the inductive sequence 
$$\xymatrix{qB_1q \ar[r]& \varphi_{2,1}(q)B_2\varphi_{2,1}(q) \ar[r]&\ldots \ar[r] &\varphi_{i,1}(q)B_i \varphi_{i,1}(q)\ar[r] & \ldots }$$
implies that the image of $q$  in the inductive limit algebra $B$ cannot be properly infinite either. Now $B$ is a simple C*-algebra, in which by a result of Cuntz in \cite{C}, every infinite projection is properly infinite. Hence, the image of $q$ in $B$ is finite.

Now, injectivity of the connecting maps $\varphi_j$ implies that $q$ must be finite, too. 
\end{proof}

\noindent We can now prove the main result.

\begin{theorem}\label{finitemultiplierproj}
Let
$$Q=\bigoplus_{j=1}^\infty p_{I_j}\ \in \mathcal{M} (C (\prod_{j=1}^\infty S^2,\KK) )$$
be a multiplier projection as before. Suppose there is some $k\in\NN$ such that $k\cdot g\npreceq Q$.

Then $Q$ is finite.
\end{theorem}

\begin{proof}
First we reduce to the case that $\NN\setminus\bigcup_{j=1}^\infty I_j$ is infinite:\\

\noindent Consider the projection map
$\rho:\prod_{j=1}^\infty S^2 \rightarrow \prod_{j=1}^\infty S^2$ onto the odd coordinates:
$$\rho(x_1,x_2,x_3,x_4,x_5,\ldots )=(x_1,x_3,x_5,\ldots).$$
Then the induced mapping $\rho^*:C(\prod_{j=1}^\infty S^2,\KK)\rightarrow C(\prod_{j=1}^\infty S^2,\KK)$ given by
$$\rho^*(f)= f\circ \rho$$
is injective and extends to an injective mapping between the multiplier algebras
$$\rho^*:\mathcal{M}\left (C (\prod_{j=1}^\infty S^2,\KK)\right)\rightarrow \mathcal{M}\left (C (\prod_{j=1}^\infty S^2,\KK )\right)$$
(to see this consult \cite{L} Proposition 2.5 and use that $\rho^*(n\cdot g)\stackrel{n\rightarrow \infty}{\rightarrow}\mathbbm{1}$, where $g$ denotes a constant one-dimensional projection as before).

Now $Q$ must be finite in $\mathcal{M}(C(\prod_{j=1}^\infty S^2,\KK))$, if $\rho^*(Q)$ is. Indeed, on supposing $Q$ to be infinite, i.e. $Q\sim Q_0<Q$ for some projection $Q_0$,  injectivity of $\rho^*$ implies $\rho^*(Q-Q_0)>0$ and hence infiniteness of $\rho^*(Q)$. 

But now $\rho^*(Q)$ is of the same form as $Q$, i.e.,
$$\rho^*(Q)=\bigoplus_{j=1}^\infty p_{\tilde{I}_j},$$
and the sets $\tilde{I}_j$ of indices being used satisfy $\NN\setminus\bigcup_{j=1}^\infty \tilde{I}_j \supseteq 2\NN$, and in particular $\NN\setminus\bigcup_{j=1}^\infty \tilde{I}_j$ is infinite, as desired.\\

\noindent After this reduction step we start the main part of the proof. By assumption we can find $k\in\NN\cup\{0\}$ such that $k\cdot g\preceq Q$, but $(k+1)\cdot g\npreceq Q$. Choose a partition $\{ A_i   \}_{i=-1}^\infty$ of $\NN$ such that each $A_i$ is infinite and such that $A_0=\bigcup_{j=1}^\infty I_j$, i.e., $A_0$ contains exactly all the indices used in our multiplier projection $Q$. Also, choose a partition $\{ B_i   \}_{i=-\infty}^\infty$ of $A_{-1}$ with each $B_i$ of cardinality $k$, except in the the case $k=0$ where we do not need the sets $B_i$ at all.

For each $r\geq 0$, choose an injective map
$$\gamma_r:\ZZ\times A_r\rightarrow A_{r+1}.$$
We can now define an injective map $\nu:\ZZ\times \NN\rightarrow \NN$, by
$$\nu(j,l)=\gamma_r(j,l)\mbox{, for every }l\in A_r.$$
Injectivity of $\nu$ follows from injectivity of each $\gamma_r$ and disjointness of the sets $A_j$.\\

\noindent Using the injective map $\nu$, let us now define a $^*$-homomorphism $$\varphi:\mathcal{M}(C(\prod_{j=1}^\infty S^2,\KK))\rightarrow \mathcal{M}(C(\prod_{j=1}^\infty S^2,\KK)).$$

The construction of this $^*$-homomorphism is only a small variation of a mapping that M. R\o rdam defined in his paper \cite{R1} to construct ``A simple C*-algebra with a finite and an infinite projection". $\varphi$ will depend on the natural number $k$ from the hypothesis of the theorem. But the change of $\varphi$ for varying $k$ is minor, so we can take care of all cases at once. (Only the case $k=0$ has to be treated separately, but this is actually exactly R\o rdam's map from \cite{R1}.)

For $j\leq 0$ and in the case $k\geq 1$ we define $\varphi_j:C(\prod_{j=1}^\infty S^2,\KK)\rightarrow C(\prod_{j=1}^\infty S^2,\KK)$ by
$$\varphi_j(f)(x_1,x_2,x_3,\ldots)=\tau(f(x_{\nu(j,1)},x_{\nu(j,2)},x_{\nu(j,3)},x_{\nu(j,4)},\ldots)\otimes p_{B_j})$$
with the finite sets $B_j\subseteq\NN$ chosen above, and a chosen isomorphism $\tau:\KK\otimes\KK\rightarrow\KK$. In the case $k=0$ we simply define $\varphi_j$ by 
$$\varphi_j(f)(x_1,x_2,x_3,\ldots)=f(x_{\nu(j,1)},x_{\nu(j,2)},x_{\nu(j,3)},x_{\nu(j,4)},\ldots).$$

For $j\geq 1$ we define $\varphi_j:C(\prod_{j=1}^\infty S^2,\KK)\rightarrow C(\prod_{j=1}^\infty S^2,\KK)$ by
$$\varphi_j(f)(x_1,x_2,x_3,\ldots)=\tau(f(c_{j,1},c_{j,2},\ldots,c_{j,j},x_{\nu(j,j+1)},x_{\nu(j,j+2)},\ldots)\otimes p_{B_j\cup\{\nu(j,1),\nu(j,2),\ldots,\nu(j,j)\}})$$
with points 
$$\begin{array}{ccccc}
c_{1,1}\\
c_{2,1} & c_{2,2}\\
c_{3,1}&c_{3,2}&c_{3,3}\\
c_{4,1}&c_{4,2}&c_{4,3}&c_{4,4}\\
\vdots & \vdots& \vdots & \vdots & \ddots\\
 \end{array}$$
in $S^2$ chosen in such a  way that for all $j\in\NN$,
$$\{(c_{k,1},c_{k,2},\ldots,c_{k,j})\ | \ k\geq j\}\mbox{ is dense in }\prod_{i=1}^j S^2.$$
(Here the case $k=0$ just means that every set $B_j$ is taken to be the empty set.)

After choosing a sequence of isometries $\{S_j\}_{j=-\infty}^\infty$ in $\mathcal{M}(C(\prod_{j=1}^\infty S^2,\KK))$ such that
$$S_j^*S_j=\mathbbm{1}\mbox{ for all $j\in\ZZ$ and }\sum_{j=-\infty}^\infty S_jS_j^*=\mathbbm{1},$$
define $\tilde{\varphi}:C(\prod_{j=1}^\infty S^2,\KK)\rightarrow \mathcal{M}(C(\prod_{j=1}^\infty S^2,\KK))$ by 
$$\tilde{\varphi}:=\sum_{j=-\infty}^\infty S_j \varphi_j S_j^*.$$
Then by Proposition \ref{exactnumberoftrivialsubprojectionsmultiple}, recalling that the cardinality of each set $B_j$ was chosen to be equal to $k$, and by the fact that $\varphi_j(g)\sim p_{B_j}$ for all $j\leq 0$, we get 
$$\tilde{\varphi}\left ((k+1)\cdot g\right )\succeq \bigoplus_{j=-\infty}^0 (k+1)\cdot p_{B_j}\succeq \bigoplus_{j=-\infty}^0g\sim \mathbbm{1}.$$
Hence $\tilde{\varphi}(n\cdot g)$ converges strictly for $n\rightarrow \infty$ to a projection $$F\sim\bigoplus_{j=-\infty}^\infty F_j\succeq \mathbbm{1},$$
where 
$$F_j=\left \{ \begin{array}{ll}\bar{\tau}(\mathbbm{1}\otimes p_{B_j})&  ,\mbox{ for }j\leq 0\mbox{ and } k\geq 1\\
 \mathbbm{1}&  ,\mbox{ for }j\leq 0 \mbox{ and } k= 0\\
\bar{\tau}\left ( \mathbbm{1}\otimes p_{B_j\cup\{\nu(j,1),\nu(j,2),\ldots,\nu(j,j)\}}\right ) &  ,\mbox{ for }j\geq 1 \mbox{ and } k\geq 1\\
\bar{\tau}\left ( \mathbbm{1}\otimes p_{\{\nu(j,1),\nu(j,2),\ldots,\nu(j,j)\}}\right ) &  ,\mbox{ for }j\geq 1 \mbox{ and } k=0.\\
\end{array}\right.$$
Here the map $\bar{\tau}:\BH\otimes \BH \rightarrow \BH$ is the extension of $\tau$ to $\BH\otimes\BH$, which exists because $\tau(e_n\otimes e_n)\stackrel{n\rightarrow \infty}{\longrightarrow} \mathbbm{1}$ strictly (\cite{L}).

Since $F\succeq \mathbbm{1}$, by Lemma 4.3 of \cite{R1} $F\sim \mathbbm{1}$ and hence there is an isometry $V\in\mathcal{M}(C(\prod_{j=1}^\infty  S^2,\KK))$ such that the map 
$$\varphi:=V^*\tilde{\varphi}V$$
extends to a unital mapping $\varphi:\mathcal{M}(C(\prod_{j=1}^\infty  S^2,\KK))\rightarrow \mathcal{M}(C(\prod_{j=1}^\infty  S^2,\KK))$. (Here we are using \cite{L} again.)

For every $0\neq f$ there is some $\delta >0$ and some open set $$\begin{array}{cccc} U&=& U_1\times U_2\times U_3\times \ldots & \times U_r\times S^2\times S^2\times \ldots \\ & \subseteq & S^2\times S^2\times S^2\times \ldots & \times S^2\times S^2\times S^2\times  \ldots\end{array}$$ such that $\|f_{|U}\|\geq \delta$. By the density condition on the $c_{ij}$ there are infinitely many $j\geq 0$ such that for any $x\in\prod_{j=1}^\infty S^2$, 
$$\|\varphi_j(f)(x)\|\geq \delta>0.$$ 
Hence $\varphi(f)(x)\in\BH\setminus \KK$ for all $x$ and $\varphi(f)\in\mathcal{M}(C(\prod_{j=1}^\infty  S^2,\KK))\setminus C(\prod_{j=1}^\infty  S^2,\KK)$.

 In particular, $\varphi$ is injective, and $C(\prod_{j=1}^\infty  S^2,\KK)\varphi(f)C(\prod_{j=1}^\infty  S^2,\KK)$ is norm dense in $C(\prod_{j=1}^\infty  S^2,\KK)$. (The latter holds since $\varphi(f)(x)\neq 0$ for all $x\in\prod_{j=1}^\infty S^2$.)

We get that $\left (k+1\right )\cdot g$ is an element in $C(\prod_{j=1}^\infty  S^2,\KK)\varphi(f)C(\prod_{j=1}^\infty  S^2,\KK)$. Further, $\varphi(\left (k+1\right )\cdot g)\succeq \mathbbm{1}$, and so $\varphi^2(f)$ is full in $\mathcal{M}(C(\prod_{j=1}^\infty  S^2,\KK))$.


This implies the simplicity of the inductive limit
$$B:=\lim_{\rightarrow}\left (\mathcal{M}\left (C(\prod_{j=1}^\infty  S^2,\KK)\right ),\varphi\right ).$$
\\
\noindent We have now arrived in the setting of Lemma \ref{finitenessargument} and it suffices to show that $\varphi^m(Q)$ is not properly infinite for all $m\in\NN$. For this we define maps
$$\alpha_j:\mathcal{P}(\NN)\rightarrow \mathcal{P}(\NN),\ j\in\ZZ,$$
$$\alpha_j(J)=\nu(j,J)\cup B_j\cup\{\nu(j,1),\nu(j,2),\ldots,\nu(j,j)\},$$
with the convention that $\{\nu(j,1),\nu(j,2),\ldots,\nu(j,j)\}=\emptyset$ for $j\leq 0$. To simplify our computations let us introduce new notation and denote from now on $B_j\cup\{\nu(j,1),\nu(j,2),\ldots,\nu(j,j)\}$ simply by $\tilde{B}_j$. 

With these definitions, one has $$\varphi(p_I)\sim\bigoplus_{j=-\infty}^\infty p_{\alpha_j(I)}=\bigoplus_{j=-\infty}^\infty p_{\nu(j,I)\cup \tilde{B}_j}.$$

Set $\Gamma_0:=\{I_s\ | \ s\in\NN\}$ and define inductively
$$\Gamma_{n+1}:=\{\alpha_j(I)\ | \ j\in\ZZ,\ I\in\Gamma_n\}.$$
Then $$\varphi^m(Q)\sim\bigoplus_{I\in\Gamma_m}p_I.$$

We will prove that $\varphi^m(Q)$ is not properly infinite by applying R\o rdam's criterion (Proposition \ref{infiniteeisnotsubprojection}), showing that for each $m\geq 1$ there is an injective map
$$t_m:\Gamma_m \rightarrow \NN$$
such that $t_m(I)\in I$ for all $I\in \Gamma_m$.
Once we have this map, it follows that 
$$\varphi^m(Q)\sim\bigoplus_{I\in\Gamma_m}p_I\nsucceq g$$
for any $m\geq 1$. But for each $m$ the projection $g$ is in the ideal of $C (\prod_{j=1}^\infty S^2,\KK)$ given by
$$\left (C (\prod_{j=1}^\infty S^2,\KK)\right )\varphi^m(Q) \left ( C(\prod_{j=1}^\infty S^2,\KK)\right ).$$ 
Then $g\preceq l\cdot \varphi^m(Q)$ for some $l\in\NN$ (\cite{RLL} Exercise 4.8) and an application of Lemma \ref{notproperlyinfinite} shows that none of the projections $\varphi^m(Q)$, $m\in\NN$, is properly infinite. By Lemma \ref{finitenessargument} this implies that the projection $Q$ is finite.\\

\noindent The maps $t_m$ are defined inductively as follows: For $m=1$, note that 
$$\Gamma_1=\{\nu(j,I_s)\cup \tilde{B}_j\ | \ j\in\ZZ,s\in\NN\}.$$
For each $j\in \ZZ$, set 
$$\Gamma_1^j:=\{\nu(j,I_s)\cup \tilde{B}_j\ | \ s\in\NN\}=:\{J_s^j\ |\ s\in\NN\}.$$
 Then
$$\Gamma_1=\bigcup_{j=-\infty}^\infty \Gamma_1^j=\{J_s^j\ |\ s\in\NN, j\in\ZZ\}\mbox{, and }\Gamma_1^{j_1}\cap\Gamma_1^{j_2}=\emptyset\mbox{ for $j_1\neq j_2$.}$$
(The latter property holds, because $\nu$ was chosen to be injective.)

Since $k\cdot g\preceq  Q$, but $(k+1)\cdot g\npreceq Q$, we know by Proposition \ref{infiniteexactnumberoftrivialsubprojections} that for any finite subset $F\subseteq \NN$
\begin{equation*}
 \left | \bigcup_{s\in F}I_s\right |+k\geq  |F|,
\tag{$\ast \ast \ast$}\end{equation*}
and in the case $k\geq 1$ that there is some finite subset $F_0$ such that
$$\left | \bigcup_{s\in F_0}I_s\right |+k=  |F_0|.$$
If $k=0$, we set $F_0$ to be the empty set.

After choosing such a finite subset $F_0$, for any finite subset $F\supseteq F_0$ we must have  
$$\left |\left  ( \bigcup_{s\in F}I_s\right ) \setminus \left ( \bigcup_{s\in F_0}I_s\right ) \right |\geq |F\setminus F_0|,$$
since, otherwise, the finite subset $F$ would violate the inequality ($\ast \ast \ast$). By injectivity of $\nu$ we get for each $j\in\ZZ$ that
$$\left | \bigcup_{s\in F_0}\nu(j,I_s)\right |+k=  |F_0|,$$
and
$$\left |\left  ( \bigcup_{s\in F}\nu(j,I_s)\right ) \setminus \left ( \bigcup_{s\in F_0}\nu(j,I_s)\right ) \right |\geq |F\setminus F_0|.$$

Then by Hall's marriage theorem one can find for each $j$ an injective mapping
$$t_1^j:\Gamma_1^j \rightarrow \NN$$
such that for all $J_s^j= \left (\nu(j,I_s)\cup \tilde{B}_j\right ) \in\Gamma_1^j$,
$$t_1^j(J_s^j)\in J_s^j\mbox{ , and }t_1^j\left (J_s^j\right ) \notin B_j \mbox{ whenever $s\notin F_0$.}$$

(Using the cardinality of each $B_j\subseteq\tilde{B}_j$, $\left | B_j\right |=k$, we are able to construct the injective map $$t_1^j:\{J_s^j\ |\ s\in F_0\} \rightarrow \{J_s^j\ |\ s\in F_0\}$$ successively in $s$, choosing different elements of $B_j$ for different values of $s$.)

By injectivity of $\nu$ and pairwise disjointness of the sets $B_j,\ j\in\ZZ,$ there is then an injective map 
$$t_1: \{J_s^j\ |\ s\in\NN, j\in\ZZ\}=\Gamma_1\rightarrow \NN.$$
We have finished defining an injective map $t_1:\Gamma_1 \rightarrow\NN$.

Inductively we define $t_{m+1}:\Gamma_{m+1}\rightarrow \NN$ after definition of $t_m:\Gamma_{m}\rightarrow \NN$ by
$$t_{m+1}(\alpha_j(I)):=\nu(j,t_m(I))$$
for $\alpha_j(I)\in\Gamma_{m+1}$ (and $I\in\Gamma_m$).

With this choice the map $t_{m+1}$ is injective. Indeed, the equations
$$\begin{array}{ccc} t_{m+1}(\alpha_j(I))&=&t_{m+1}(\alpha_{\tilde{j}}(\tilde{I}))\\ \shortparallel & &\shortparallel\\ \nu(j,t_m(I))& &\nu(\tilde{j},t_m(\tilde{I}))\end{array}$$
imply by injectivity of $\nu$ that
$$ j=\tilde{j}\mbox{, and }t_m(I)=t_m(\tilde{I}).$$
By the induction hypothesis, $t_m$ was chosen to be injective, and hence
$$I=\tilde{I}.$$
\\
For each $m\in\NN$ we ended up with an injective map $$t_m:\Gamma_m\rightarrow \NN$$
such that $t_m(I)\in I$ for all $I\in \Gamma_m$, which is all that was left to construct.
\end{proof}

\begin{corollary}\label{nonfullequalsstablyfinite}
Let Let $$Q=\bigoplus_{j=1}^\infty p_{I_j}\in \MCZK.$$

Then $Q$ is non-full if, and only if, $Q$ is stably finite.
\end{corollary}

\begin{proof}
If all multiples $n\cdot Q$ of $Q$ are finite, then $n\cdot Q\nsucceq \mathbbm{1}$ for any $n\in\NN$ and $Q$ can't be full. The converse direction follows from combining Theorem \ref{finitemultiplierproj} with Theorem \ref{nonfull}.
\end{proof}

\noindent If a multiplier projection of the form $$Q=\bigoplus_{j=1}^\infty p_{I_j}$$ is full, then $\mathbbm{1}\preceq m\cdot Q$ for some $m\in\NN$. Hence some multiple of $Q$ is properly infinite. The projection $Q$ itself might be finite though (see \cite{R1}).

On the other hand if $Q$ is non-full, then $Q$ is stably finite by Corollary \ref{nonfullequalsstablyfinite}.\\

\noindent Summarized, the results state that  every multiplier projection in $\mathcal{M}(C(\prod_{j=1}^\infty S^2,\KK))$ of the special form (*) considered above is either non-full and stably finite , or full and stably properly infinite. \\

{
\small
\bibliographystyle{ieeetr}

\end{document}